\documentclass{amsart}
\usepackage{amsfonts,amssymb,amscd,amsmath,enumerate,verbatim,calc}
\usepackage[all]{xy}

\newcommand{\CM}{Cohen-Macaulay}

\newcommand{\wrt}{with respect to}

\newcommand{\m}{\mathfrak{m} }
\newcommand{\M}{\mathfrak{M} }
\newcommand{\N}{\mathfrak{N} }

\newcommand{\q}{\mathfrak{q} }
\newcommand{\p}{\mathfrak{p} }

\newcommand{\A}{\mathfrak{a} }
\newcommand{\F}{\mathcal{F} }
\newcommand{\Z}{\mathbb{Z} }

\newcommand{\Oo}{\mathcal{O} }
\newcommand{\rt}{\rightarrow}

\newcommand{\ov}{\overline}

\newcommand{\depth}{\operatorname{depth}}

\newcommand{\Tr}{\operatorname{Tr}}

\newcommand{\rank}{\operatorname{rank}}
\newcommand{\Hom}{\operatorname{Hom}}
\newcommand{\Ext}{\operatorname{Ext}}

\theoremstyle{plain}

\newtheorem{theorem}{Theorem}[section]

\newtheorem{lemma}[theorem]{Lemma}
\newtheorem{proposition}[theorem]{Proposition}

\newtheorem{conjecture}[theorem]{Conjecture}

\theoremstyle{definition}

\newtheorem{s}[theorem]{}
\newtheorem{remark}[theorem]{Remark}

\theoremstyle{remark}

\begin{document}

\title[invariant rings]{The Cohen-Macaulay property of invariant rings over ring of integers of a global field}
\author{Tony~J.~Puthenpurakal}
\date{\today}
\address{Department of Mathematics, IIT Bombay, Powai, Mumbai 400 076}

\email{tputhen@math.iitb.ac.in}


 \begin{abstract}
Let $A$ be the ring of integers of global field $K$.  Let $G \subseteq GL_2(A)$ be a finite group. Let $G$ act linearly on $R = A[X,Y]$ (fixing $A$). Let $R^G$ be the ring of invariants. In the equi-characteristic case we prove $R^G$ is \CM. In mixed characteristic case we prove that if for all primes $p$ dividing $|G|$ the Sylow $p$-subgroup of $G$ has exponent $p$ then $R^G$ is \CM. We prove a similar case if for all primes $p$ dividing $|G|$ the prime $p$ is un-ramified in $K$.
\end{abstract}
 \maketitle
\section{introduction}
Let $K$ be  a field and let $G \subseteq GL_n(K)$ be  a finite group acting linearly on $R = K[X_1, \ldots, X_n]$. Let $R^G$ be the ring of invariants.
 The study of such invariant  rings have a rich history both in the non-modular case (when $|G|$ is invertible in $K$) and in the modular case
(when $|G| = 0$ in $K$). See \cite{S} for a nice book on the subject.

\s \label{intro} In this paper we assume $A$ is ring of integers in a global field, i.e., $A$ is one of the following two rings
\begin{enumerate}
  \item  the ring of integers of a number field $K$ (i.e., $K$ is a finite extension of $\mathbb{Q}$).
  \item the  ring of integers of finite extension of $F_q(t)$ (where $F_q$ is a finite field with $q$ elements).
\end{enumerate}
Let $G \subseteq GL_n(A)$ be a finite group. Let $R = A[X_1, \ldots, X_n]$ and let $G$ act linearly on $R$ (fixing $A$). In this paper we study \CM \ property
 of $R^G$ when $n = 2$. One might wonder why should one study such rings. As we show that there is nice interplay between algebraic number theory, group cohomology
 and commutative algebra while investigating such questions.
Previously the case when $n = 2, 3$ was studied when $A = \Z$, see \cite{A}. In this paper  it was proved that $\Z[X_1, X_2]^G$ is \CM  \ while $\Z[X_1, X_2, X_3]^G$ need not be \CM.

\s \label{setup} Consider the ring $R = A[X,Y]$. Let $G \subseteq GL_2(A)$ be a finite group. $G$ acts linearly on $R$ (fixing $A$). Let $R^G$ be the ring of invariants. Clearly $R^G$ is a normal domain of dimension three. A natural question is when $R^G$ is \CM.

In this paper we prove the following two results.
In the equi-characteristic case (when $A$ is of type (2) in \ref{intro}) we show
\begin{theorem}\label{equi}
(with hypotheses as in \ref{setup}). Further assume $A$ is the ring of integers of a finite extension of $F_q(t)$. Then $R^G$ is \CM.
\end{theorem}
In the mixed characteristic case we prove the following:
\begin{theorem}\label{mixed}
(with hypotheses as in \ref{setup}). Further assume $A$ is the ring of integers of  a number field $K$. Assume for every prime $p$ dividing $|G|$, the Sylow $p$-subgroup of $G$ has exponent $p$. Then $R^G$ is \CM.
\end{theorem}
Theorem \ref{mixed} has condition on the group $G$. The next result gives the \CM \ property of $R^G$ under a condition on $K$.
\begin{theorem}\label{unramify}
(with hypotheses as in \ref{setup}). Further assume $A$ is the ring of integers of  a number field $K$. Assume for every prime $p$ dividing $|G|$, the prime $p$ is un-ramified in $K$. Then $R^G$ is \CM.
\end{theorem}

Perhaps the simplest groups to be considered are cyclic groups. In this case we prove the following result:
\begin{theorem}\label{cyclic-thm}
(with hypotheses as in \ref{setup}). Further assume $A$ is the ring of integers of  a number field $K$. Assume $G = \ <\sigma> \  \subseteq GL_2(A)$ is  cyclic.
Let $\epsilon_1, \epsilon_2$ be eigen-values of $\sigma$.
 Assume one of the following two conditions holds.
\begin{enumerate}[\rm (1)]
\item
$\epsilon_1 = \epsilon_2$.
\item
 $\epsilon_2 = \epsilon_1^{-1}$ and either $|G|$ is odd or $|G| = 2m$ with $m$ odd.
\end{enumerate}
Then $R^G$ is \CM.
\end{theorem}
We make the following:
\begin{conjecture}
Let $A$ be the  ring of integers in a number field and let $G$ be a finite subgroup of $GL_2(A)$ acting linearly on
$R = A[X,Y]$ (and fixing $A$). Then $R^G$ is \CM.
\end{conjecture}
 We now describe in brief the contents of this paper.
In section two we discuss some preliminaries on group cohomology that we need. In section three we prove some general facts about ring of invariants that we
need. We prove Theorem \ref{mixed} by induction on order of $G$. The intermediate rings obtained in induction
are no longer polynomial rings. In section 4 we describe a class of rings to enable this induction. In section five we describe
Ellingsrud-Skjelbred spectral sequences and describe an application which is crucial for us. In section
six we prove Theorem \ref{equi}. In section seven we describe our strategy to prove Theorem \ref{mixed}.
In the next two sections we describe the two cases that essentially arise in our proof of Theorem \ref{mixed}.
In section ten we prove Theorem \ref{mixed}. In section eleven we prove Theorem \ref{unramify}.
Finally in section twelve we prove Theorem \ref{cyclic-thm}.
\section{Some preliminaries on group cohomology}
In this section we discuss some preliminaries on group cohomology that we need. We believe all the assertions in this section are already known (perhaps even more generally than what is given in this section. We give a few proofs due to lack of a suitable reference.
\s  \label{setup-gc-intro}Let $S$ be a commutative Noetherian ring. 
 Let $G$ be a finite group. Let $S[G]$ be the group ring and let $Mod(S[G])$ be the category of left $S[G]$-modules. Let $(-)^G$ be the functor of $G$-fixed points. Let $H^n(G,-)$ be the $n^{th}$ right derived functor of
$(-)^G$. We call $H^n(G, M)$ the \emph{$n^{th}$-cohomology group} of $G$ with coefficients in $M$). Note this cohomology module also depends on $S$ which we have suppressed. We note that
$H^n(G, M) = \Ext^i_{S[G]}(S, M)$.

\begin{remark}
(1) If $M$ is a finitely generated  left $S[G]$-module then $H^n(G, M)$ are finitely  generated  $S$-modules. This can be easily seen by taking a graded free resolution of $S$ consisting of finitely generated free   $S[G]$-modules.

(2) We note that $\Z[G]\otimes_\Z S = S[G]$.

(3) $S[G]$ is a two sided Noetherian ring,
\end{remark}

We first discuss the behaviour of group cohomology with respect to localization.
\begin{proposition}
\label{localization}(with hypotheses as in \ref{setup-gc-intro})  Let $M$ be a $S[G]$-module.  Let $W$ be a multiplicatively closed subset of $S$. Then $W^{-1}M$ is a $W^{-1}S[G]$-module and $W^{-1}H^i(G, M) \cong H^i(G, W^{-1}M)$ for all $i \geq 0$.
\end{proposition}
\begin{proof}
We note that $S$ is contained in the center of $S[G]$.
  It is clear that  $W^{-1}M$ is a $W^{-1}S[G]$-module.  Let $\mathbb{F}$ be a free resolution of $S$ with finitely generated free $S[G]$-modules. Then it is clear that
  $W^{-1}\mathbb{F}$ is a free resolution of $W^{-1}S$ with finitely generated free $W^{-1}S[G]$-modules.
  It follows that
  $W^{-1}H^i(G, M) \cong H^i(G, W^{-1}M)$ for all $i \geq 0$.
\end{proof}

\begin{remark}\label{grading}
(1) Let $S = \bigoplus_{n \geq 0}S_n$ be  a $\mathbb{N}$-graded Noetherian ring.  Note $S[G]$ is also $\mathbb{N}$-graded (with $\deg \sigma = 0$ for all $\sigma \in G$) Noetherian ring.

(2)  If $M$ is a finitely generated graded left $S[G]$-module then $H^n(G, M)$ are finitely generated graded $S$-module. This can be easily seen by taking a graded free resolution of $S$ consisting of finitely generated free  graded $S[G]$-modules.

(3) If $\mathbb{F} $ is a free resolution of $\Z$ as a $\Z[G]$-module consisting of finitely generated free $\Z[G]$-modules then $\mathbb{G} = \mathbb{F}\otimes_\Z S$ is a graded free resolution of $S$ as an $S[G]$-module. We note that $\mathbb{G}_i$ for all $i$ is isomorphic to finitely many copies of $S[G]$ (as graded $S[G]$-module).  Note there are no shifts involved in $\mathbb{G}_i$.

(4) We note that if $M = \bigoplus_{n \in \Z}M_n$ is a $S[G]$-module then each $M_n$ is a $S_0[G]$-module.

\end{remark}

The next result is needed later in the paper.
\begin{proposition}
\label{non-neg} Let $S = \bigoplus_{n \geq 0}S_n$ be a graded ring. Let $S[G]$ be graded as above. Let $M = \bigoplus_{n \geq 0}M_n$ be a graded $S[G]$-module. Then $H^i(G, M)_n = 0$ for $n < 0$ for all $i \geq 0$.
\end{proposition}
\begin{proof}
  Let $\mathbb{G}$ be the free resolution of $S$ as a $S[G]$-module (given in \ref{grading}(3)). Then note that $H^i(G,M)$ is a sub-quotient of finitely many copies of $M$. The result follows.
\end{proof}

\s \label{cyclic} Now assume $G = <\sigma>$ be cyclic of order $m \geq 2$.  Set $\Tr = 1 + \sigma + \cdots + \sigma^{m-1}$.
(1) By \cite[6.2.1]{W} we have the following  $2$-periodic free resolution of $\Z$ as a $\Z[G]$-module:
$$ \cdots \Z[G] \xrightarrow{\Tr} \Z[G] \xrightarrow{\sigma -1} \Z[G]  \xrightarrow{\Tr} \Z[G] \xrightarrow{\sigma -1} \Z[G]  \xrightarrow{\epsilon} \Z \rt 0.$$

(2) Let $S =\bigoplus_{n \geq 0}S_n$ be a  Noetherian graded ring. As discussed in \ref{grading}(3) we have the following resolution of $S$ as an $S[G]$-module
$$ \cdots S[G] \xrightarrow{\Tr} S[G] \xrightarrow{\sigma -1} S[G]  \xrightarrow{\Tr} S[G] \xrightarrow{\sigma -1} S[G]  \xrightarrow{\epsilon} S \rt 0.$$

(3) If $M$ is a graded $S[G]$-module then note that $H^i(G, M)$ is the $i^{th}$-cohomology of the complex
$$ 0 \rt M \xrightarrow{\sigma -1} M \xrightarrow{\Tr} M \xrightarrow{\sigma -1} M \xrightarrow{Tr} \cdots. $$

As an immediate consequence of \ref{cyclic}(3) we get the following result.
\begin{proposition}\label{comp}
  Let $G = < \sigma >$ be a finite cyclic group of order $m \geq 2$. Let $M  = \bigoplus_{n \in \Z} M_n$ be  a graded $S[G]$-module. Let $H^i(G, M_n)$ be the $i^{th}$- group cohomology of $M_n$ considered as a $S_0[G]$-module.
  Then $H^i(G, M) = \bigoplus_{n \in \Z}H^i(G, M_n).$
\end{proposition}

The following result is needed later.
\begin{proposition}\label{trivial}
  Let $G = < \sigma >$ be a finite cyclic group of order $p$ ( a prime). Let $M  = \bigoplus_{n \in \Z} M_n$ be  a graded $S[G]$-module. Assume the action of $G$ on $M$ is trivial
  and $p M = 0$
  Then $H^i(G, M) = M$ for all $i \geq 0$.
\end{proposition}
\begin{proof}
  The maps $\sigma - 1 \colon M \rt M$ and $\Tr \colon M \rt M$ are zero. The result follows.
\end{proof}

\section{Some preliminaries on rings of invariants}
In this section we prove several preliminary results on rings of invariants that we need.

\s\label{hilb} We first note that if $G \subseteq GL_n(\mathbb{C})$  is a finite group then it is conjugate to group $H \subseteq GL_n(A)$ where $A$ is the ring of integers in Hilbert class field of $\mathbb{Q}(e^{2\pi i/m})$ where $m$ is the exponent of $G$.
To see this note that as $G$ has exponent $m$, $\mathbb{Q}(e^{2\pi i/m})$  is a splitting field of $G$, see \cite[41.1, p.\ 292]{CR}. It follows $G$ is conjugate to a subgroup of $GL_n(\mathbb{Q}(e^{2\pi i/m}))$.  The result now follows from
  \cite[75.4, p.\ 518]{CR}.  We note that the proof in \cite{CR} only requires an extension $K$ of $\mathbb{Q}(e^{2\pi i/m})$ such that if $I$ is any ideal in $\mathbb{Z}(e^{2\pi i/m})$ then
  $IA$ is principal in the ring of integers $A$ of $K$. This is proved in \cite[20.14]{CR}. However we need that we can choose $K$ to be the Hilbert class field of $\mathbb{Q}(e^{2\pi i/m})$
  as we want $K$ to be an un-ramified extension of $\mathbb{Q}(e^{2\pi i/m})$, see \cite[p.\ 224]{L-ant}.

  The following result is essential in our proofs.
 \begin{theorem}\label{sylow}
 Let $A$ be the ring of integers of a global field. Let $G$ be a finite subgroup of $GL_n(A)$ acting linearly on $R = A[X_1, \ldots, X_n]$ (fixing $A$). For each prime $p$ dividing $|G|$ choose a Sylow $p$-subgroup $H_p$ of $G$.
 If $R^{H_p}$ is \CM\ for all $p$ dividing $|G|$ then $R^G$ is \CM.
 \end{theorem}

 The proof of Theorem \ref{sylow} is essentially given in \cite[8.3.1]{S}. However the details are a little different. So we give the complete proof.  We also note that all Sylow subgroups of $G$ are conjugate. So if $R^{H_p}$ is \CM\ for one Sylow-$p$ subgroup  then  $R^K$ is \CM \ for any Sylow $p$-subgroup $K$ of $G$.

 \s\label{transfer}
 Let $(\Oo, \pi)$ be a DVR such that $\Oo/(\pi)$ is a field of characteristic $p$. Let $G$ be a finite subgroup of $GL_n(\Oo)$ acting linearly on $R = A[X_1, \ldots, X_n]$ (fixing $\Oo$). Let $H$ be a $p$-Sylow subgroup of $G$. Define
 \begin{align*}
  \psi^G_H &\colon R^H \rt R^G, \\
  r  &\mapsto \frac{1}{|G \colon H|}\sum_{gH \in G/H} gr.
 \end{align*}
See \cite[2.4]{S} which shows that the above definition  is independent of the choice of elements representing the cosets $G/H$. Furthermore note that $p$ does not divide $|G \colon H|$ and so it is a unit in $\Oo$. Furthermore the following properties of $\psi^G_H$ are established in \cite[2.4]{S}:
\begin{enumerate}
 \item $\psi$ is $R^G$-linear.
 \item $\psi^G_H$ is a splitting of the inclusion $R^G \hookrightarrow R^H$.
\end{enumerate}

\s Recall we say a graded ring $S = \bigoplus_{n \geq 0}S_n$ to be $*$-local if it has a unique graded maximal homogeneous ideal $\M$. Note in the case under consideration $S_0$ is local (say with maximal ideal $\m_0$) and  $\M = \m_0 + S_+$. It is well-known that $S$ is \CM \ if and only if $S_\M$ is \CM \  (this is exercise 2.1.27 in \cite{BH}). In our case $\Oo[X,Y]$ and $\Oo[X, Y]^G$ are both $*$-local.
 Next we show
 \begin{lemma}
  \label{sylow-prelim}(with hypotheses as in \ref{transfer}). If $R^H$ is \CM \ then $R^G$ is \CM.
 \end{lemma}
 \begin{proof}
  Let $\M_G$ be the $*$-maximal ideal of $R^G$. Notice as $R^H$ is a finite $R^G$-module we get $\sqrt{\M_G} = \M_H$ the $*$-maximal ideal of $R^H$. As $R^H$ is \CM \ it follows that $H^i_{\M_G}(R^H) = 0$ for $i \leq n$. We have a sequence of $R^G$-linear maps
  $$ R^G   \hookrightarrow R^H \xrightarrow{\psi^G_H} R^G $$
  such that the composite is identity. As local-cohomology is $R^G$-linear we get
  a sequence of $R^G$-linear maps for all $i$
  $$ H^i_{\M_G}(R^G)  \rightarrow H^i_{\M_G}(R^H) \xrightarrow{H^i(\psi^G_H)} H^i_{\M_G}(R^G) $$
  such that the composite is identity. It follows that $H^i_{\M_G}(R_G) = 0$ for $i \leq n$. Thus $R^G$ is \CM.
 \end{proof}

We now give
\begin{proof}
[Proof of Theorem \ref{sylow}] Let $\M$ be a graded  maximal ideal of $R^G$. Then note that $\M = P \oplus R^G_+ $ where $P$ is a maximal ideal of $\A$. We wish to show $R^G_\M$ is \CM.  Recall $(R^G)_0 = A$. By \ref{localization}  we have $(R^G)_P  \cong (R_P)^G$. It suffices to prove $(R_P)^G$ is \CM. This follows from Lemma \ref{sylow-prelim}.
\end{proof}

\s \label{base-change} Next we show that \CM \ property of ring of invariants behave well \wrt \ base change. More precisely let $A \subseteq B$ be two ring of integers of global fields such that $B$ is a finite $A$-module. We note that $GL_n(A) \subseteq GL_n(B)$. So if $G$ is a subgroup of $GL_n(A)$ then it can be naturally considered as a subgroup of $GL_n(B)$. Let $G$ be a finite subgroup of $GL_n(A)$. Let $R = A[X_1, \ldots, X_n]$ and let $G$ act linearly on $R$. Also let $T = B[X_1, \ldots, X_n]$ and let $G$ act linearly on $T$. We assert
\begin{proposition}\label{bc-prop}
$R^G$ is \CM \ if and only if $T^G$ is \CM.
\end{proposition}
\begin{proof}
We note that $T^G$ is a finitely generated $R^G$-module.
Let $\M$ be a maximal homogeneous ideal of $R^G$ and let $\N$ be a maximal ideal of $T^G$ with $\M T^G \subseteq \N$. We show $R^G_\M$ is \CM \ if and only if $T^G_\N$ is \CM. This will prove the result, see  exercise 2.1.27 in \cite{BH}. Let $P = \M \cap A$. Then note that $B_P$ is a finite free $A_P$-module. It follows that $T^G_P$ is a finite free $R^G_P$-module. Thus $R^G_\M \rt T^G_\N$ is a flat local extension. Furthermore note the fiber of this flat extension is zero dimensional. It follows that $R^G_\M$ is \CM \ if and only if $T^G_\N$ is \CM, see \cite[Corollary, p.\ 181]{M}. The result follows.
\end{proof}
To prove Theorem \ref{mixed} we need the following result:
\begin{theorem}\label{red-hilb}
Let $K$ be  the Hilbert class field of $\mathbb{Q}(e^{2\pi i/p})$ where $p$ is a prime and let $\Oo$ be its ring of integers. Suppose the following assertion holds:
for any $p$-group $H \subseteq GL_2(\Oo)$ of exponent $p$ the ring $\Oo[X, Y]^H$ is \CM. Then  the following assertion holds:
Let $L$ is any number field and let $A$ be its ring of integers. Let $G \subseteq Gl_2(A)$ be a finite $p$-group of exponent $p$. Then $A[X, Y]^G$ is \CM.
\end{theorem}
\begin{proof}
  Let $E$ be the composite of $K$ and $L$ and let $B$ be its ring of integers. By Proposition  \ref{bc-prop} it suffices to prove $B[X, Y]^G$  is \CM. By \ref{hilb} $G \subseteq Gl_2(B)$ is conjugate to a subgroup $H \subseteq Gl_2(\Oo)$. By our assumption $\Oo[X,Y]^H$ is \CM. By \ref{bc-prop} we get that $B[X, Y]^H$ is \CM. As $G$ is conjugate to $H$ we get that $B[X,Y]^G$ is \CM.
\end{proof}
\section{A class of graded rings}
We prove Theorem \ref{mixed} by induction on order of $|G|$. The intermediate rings obtained in induction are no longer polynomial rings. We need to work with a more general class of rings to enable the induction.

\s \label{class} Let $(\Oo, (\pi))$ be a DVR. By $\F_{\Oo}$ we denote a class of rings $R$ with following properties.
\begin{enumerate}
  \item $R  = \bigoplus_{n \geq 0}R_n$ is a graded \CM \ ring with $R_0 = \Oo$.
  \item $\dim R = 3$.
  \item $R$ and $R/\pi R$ are domains.
  \item Set $a(R) = \max \{ n \mid H^2_{R_+}(R)_n \neq 0 \}$. Then $a(R) \leq -2$.
\end{enumerate}

Example: The polynomial ring $\Oo[X, Y] \in \F_\Oo$. Later we prove that appropriate invariant rings are also in $\F_{\Oo}$.
The following result is useful.
\begin{proposition}\label{basic}
Let $R \in \F_\Oo$. Then $H^i_{R_+}(R) = 0$ for $i \neq 2$ and $H^2_{R_+}(R) \neq 0$.
\end{proposition}
\begin{proof}
We note that $R/\pi R$ is a two dimensional finitely generated algebra over a field $\Oo/(\pi)$. It is also \CM. So we have $H^i_{(R/\pi R)_+}(R/\pi R) = 0$ for $i \neq 2$ and $H^2_{(R/\pi R)_+}(R/\pi R) \neq  0$. We have an exact sequence
$$0 \rt R \xrightarrow{\pi} R \rt R/\pi R \rt 0.$$
Taking cohomology  with respect to $R_+$ we get that for $i \neq  2$ and for all $n \in \Z$ we have
$$H^i_{R_+}(R)_n \xrightarrow{\pi} H^i_{R_+}(R)_n \rt 0. $$
As $H^i_{R_+}(R)_n$ is finitely generated $\Oo$-module, see (see \cite[15.1.5]{bs}), it follows from Nakayama's Lemma that $H^i_{R_+}(R)_n = 0$ for all $n \in \Z$ (when $i \neq 2$).

We also have a sequence
$$H^2_{R_+}(R) \rt H^2_{R_+}(R/\pi R) \rt H^3_{R_+}(R) = 0. $$
It follows that $H^2_{R_+}(R) \neq 0$.
\end{proof}

\s\label{action}\emph{Group actions:} Let $R\in \F_\Oo$. Let $Aut^*(R)$ be the set of automorphism $\sigma \colon R \rt R$ such that
\begin{enumerate}
  \item $\sigma(R_n) \subseteq R_n$ for all $n \geq 0$. Furthermore $\sigma$ on $R_0 = \Oo$ is identity,
  \item The map $R_n \rt R_n$ induced by $\sigma$ is $\Oo$-linear.
\end{enumerate}
We consider only finite subroups of $Aut^*(R)$. We also note that if $\sigma \in Aut^*(R)$ then $\sigma$ induces an automorphism on $R/\pi R$.

\s\label{les} Let $G \subseteq Aut^*(R)$ and let $S = R^G$. Clearly $S$ is a Noetherian domain of dimension three. Furthermore $R$ is a finite $S$-module.
We have an exact sequence of $S[G]$-modules
  $$ 0 \rt R \xrightarrow{\pi} R \rt R/\pi R \rt 0.$$
  Taking invariants we have
  an exact sequence
  \begin{align*}
    0 &\rt S \xrightarrow{\pi} S \rt (R/\pi R)^G \\
    &\rt H^1(G, R) \xrightarrow{\pi}  H^1(G, R)  \rt H^1(G, R/\pi R)  \\
    &\rt H^2(G, R) \xrightarrow{\pi}  H^2(G, R)  \rt H^2(G, R/\pi R) \\
    &\cdots
  \end{align*}

 We prove the following preliminary results of $S$.
\begin{proposition}\label{prelim}
(with hypotheses as above). We have
\begin{enumerate}[\rm (1)]
  \item $S/\pi S$ is a domain of dimension two.
  \item  $H^i_{S_+}(S) = 0$ for $i \geq 3$.
  \item $\sqrt{S_+ R} = R_+$.
  \item $S$ is \CM \ if and only if $\depth_\M H^1(G, R)  > 0$ (where $\M$ is the $*$-maximal ideal of $S$.
\end{enumerate}
\end{proposition}
\begin{proof}
  (1) By \ref{les} we have an exact sequence
  $$ 0 \rt S \xrightarrow{\pi} S \rt (R/\pi R)^G.$$
  As  $(R/\pi R)^G$ is  a domain it follows that $S/\pi S$ is a domain. Also as $S$ is three dimensional *-local ring and $\pi$ is in its *-maximal ideal is $S$-regular, it follows that $\dim S/\pi S = 2$.

  (2) We note that $(S/\pi S)_0 = \Oo/(\pi)$ is a field. As $\dim S/\pi S = 2$ it follows that $H^i_{S_+}(S/\pi S) = 0$ for $i \geq 3$.
  We have an exact sequence
$$0 \rt S \xrightarrow{\pi} S \rt S/\pi S \rt 0.$$
Taking cohomology  with respect to $S_+$ we get that for $i \geq  3$ and for all $n \in \Z$ we have
$$H^i_{S_+}(S)_n \xrightarrow{\pi} H^i_{S_+}(S)_n \rt 0. $$
As $H^i_{S_+}(S)_n$ is finitely generated $\Oo$-module, (see \cite[15.1.5]{bs}), it follows from \\ Nakayama's Lemma that $H^i_{S_+}(S)_n = 0$ for all $n \in \Z$ (when $i \geq 3$).

(3) Let $u, v \in S_+$ be homogeneous such that their images in $S/\pi S$ is a homogeneous system of parameters of $S/\pi S$. Set $T = S/(u, v)$. Then as $(T/\pi T)_n  =0$ for $n \gg 0$ it follows that $T_n = 0$ for $n \gg 0$. As $R$ is a finite $S$-module it follows that $(R/(u, v)R)_n = 0$ for $n \gg 0$.  Thus $R_+ \subseteq \sqrt{(u, v)R}$. So $R_+ \subseteq \sqrt{S_+ R}$.
As $S_+R \subseteq R_+$ it follows that $\sqrt{S_+R} \subseteq R_+$. The result follows.

(4)
Note as $S$ is $*$-local, $S$ is \CM \ if and only if $S_\M$ is \CM.
By \ref{les} we have an exact sequence $0 \rt S/\pi S \rt (R/\pi R)^G \rt C \rt 0$ and  a exact sequence $0 \rt C \rt H^1(G, R) \xrightarrow{\pi} H^1(G, R)$.
Note as $\depth_\M R/\pi R \geq 2$ it is elementary to show that $\depth_\M (R/\pi R)^G \geq 2$.

If $\depth_\M H^1(G, R) > 0$ then $\depth_\M C > 0$. By depth Lemma it follows that $\depth_\M S/\pi S \geq 2$. It follows that $\depth_\M S \geq 3$. So $S$ is \CM.

If $S$ is \CM \ then it follows that $\depth_\M S/\pi S = 2$ and so $\depth_\M  C \geq 1$.  We have an exact sequence
$$ 0 \rt \Hom_S(S/\M, C) \rt \Hom_S(S/\M, H^1(G, R)) \xrightarrow{\pi} \Hom_S(S/\M, H^1(G, R)). $$
As $\pi \in \M$ the middle map above is zero. So we get
$$ \Hom_S(S/\M, C) \cong  \Hom_S(S/\M, H^1(G, R)). $$
As $\depth_\M C > 0$ we get that $ \Hom_S(S/\M, C) = 0$. The result follows.
\end{proof}
\section{ Ellingsrud-Skjelbred spectral sequences and an application}
In this section we describe the Ellingsrud-Skjelbred spectral sequences, see \cite{ES} (we follow the exposition given in \cite[8.6]{Lorenz}). We also give an application which is crucial for us.

\s Let $S$ be a commutative Noetherian ring. Let $Mod(A)$ be the category of left $A$-modules.

(1) Let $\A$ be an ideal in $S$. Let $\Gamma_\A(-)$ be the torsion functor associated to $\A$.  Let $H^n_\A(-)$ be the $n^{th}$ right derived functor of
$\Gamma_\A(-)$. 

(2) If $M \in Mod(S[G])$ then note $\Gamma_\A(M) \in Mod(S[G])$.

\s Ellingsrud-Skjelbred spectral sequences  are constructed as follows: Consider the following sequence of functors
\[
(i) \quad   Mod(S[G]) \xrightarrow{(-)^G} Mod(S) \xrightarrow{\Gamma_\A} Mod(S),
\]
\[
(ii)   \quad   Mod(S[G]) \xrightarrow{\Gamma_\A} Mod(S[G]) \xrightarrow{(-)^G} Mod(S)
\]
We then notice
\begin{enumerate}[\rm (a)]
\item
The above compositions are equal.
\item
It is possible to use Grothendieck spectral sequence of composite of functors to both (i) and (ii) above; see
\cite[8.6.2]{Lorenz}.
\end{enumerate}
Following Ellingsrud-Skjelbred  we let $H^n_\A(G,-)$ denote the $n^{th}$ right derived functor of this composite
functor. So by (i) and (ii) we have two first quadrant spectral sequences for each $S[G]$-module $M$
\[
(\alpha)\colon \quad \quad  E_2^{p,q} = H^p_\A(H^q(G, M)) \Longrightarrow H^{p+q}_\A(G,M), \ \text{and}
\]
\[
(\beta)\colon \quad \quad \mathcal{E}_2^{p,q} = H^p(G, H^q_\A(M)) \Longrightarrow H^{p+q}_\A(G,M).
\]

\begin{remark}\label{grading-es}
Ellingsrud-Skjelbred spectral sequences have an obvious graded analogue.
\end{remark}

\s \textit{Application:} Let $(\Oo, \pi)$ be a DVR. Let $R \in \F_\Oo$, see \ref{class}.  Let $G \subset Aut^*(R)$ be a finite group acting  on $R$ (see \ref{action}). Let $S = R^G$ be the ring of invariants. We consider Ellingsrud-Skjelbred spectral sequences  with $\A = S_+$ and $M = R$.  We first prove
\begin{proposition}\label{second-es}
The spectral sequence $\beta$ collapses at $E_2$-stage. We have
\begin{enumerate}[\rm (1)]
  \item $H^0_{S_+}(G, R) = H^1_{S_+}(G, R) = 0$.
  \item For $i \geq 2$ we have $H^i_{S_+}(G, R) = H^{i-2}(G, H^2_{S_+}(R))$.
  \item For all $i$,  $H^i_{S_+}(G, R)_j = 0$ for $j \geq -1$.
\end{enumerate}
\begin{proof}
  We note that as $\sqrt{S_+ R} = R_+$ we get $H^q_{S_+}(R)$ vanishes except at $q = 2$, see \ref{basic}. It follows that $\beta$ collapses at $E_2$ stage. The assertions (1) and (2) follow immediately. To see
  (3) note that the action of $G$ preserves degrees and $H^2_{S_+}(R)_j = 0$ for $j \geq -1$ (see \ref{class}(4)).
\end{proof}
\end{proposition}
Next we analyze the spectal sequence $\alpha$.
We show
\begin{proposition}\label{first-es}
We have graded exact sequences
\begin{enumerate}[\rm (1)]
  \item $0 \rt H^0_{S_+}(H^1(G, R)) \rt H^2_{S_+}(S) \rt L \rt 0$ where $L_j = 0$ for $j \geq -1$.
  \item $0 \rt H^0_{S_+}(H^3(G, R)) \rt H^2_{S_+}(H^2(G, R))$.
\end{enumerate}
\end{proposition}
\begin{proof}
  (1) By the spectral sequence $\alpha$ we get an exact sequence
  $$ 0 \rt E_3^{0, 1} \rt E_2^{0, 1} \rt E_2^{2,0} \rt E_3^{2,0} \rt 0.$$
  Here $E^{0,1}_2 = H^0_{S_+}(H^1(G, R))$ and $E_2^{2,0} = H^2_{S_+}(S)$. Notice  $E_3^{0, 1} =  E_\infty^{0, 1}$ which being a sub-quotient of $H^1_{S_+}(G, R) = 0$.
  We also have $  E_3^{2,0} =  E_\infty^{2,0}$. The later module  is a sub-quotient of $H^2_{S_+}(G, R)$ which is concentrated in degrees $\leq -2$, see \ref{second-es}. The result follows.

  (2) By the spectral sequence $\alpha$ we get an exact sequence
  $$ 0 \rt E_3^{0, 3} \rt E_2^{0, 3} \rt E_2^{2,2}$$
  Here $E^{0,3}_2 = H^0_{S_+}(H^3(G, R))$ and $E_2^{2,2} = H^2_{S_+}(H^2(G, R))$.
  As $H^i_{S_+}(-) = 0$ for $i \geq 3$ it follows that
   $E_3^{0, 3} =  E_\infty^{0, 3}$. The later module is a sub-quotient of $H^3_{S_+}(G, R)$ which is concentrated in degrees $\leq -2$, see \ref{second-es}.
  We also have $E_3^{0,3}$ is a submodule of $H^0_{S_+}(H^2(G, R))$ which is concentrated in degrees $\geq 0$, see \ref{non-neg}. Thus $E_3^{0,3} = 0$. The result follows.
\end{proof}
The following result is needed later.
\begin{proposition}\label{incl}
  Let $\Oo$ be a DVR of mixed characteristic. Let $R \in \F_\Oo$, see \ref{class}.  Let $G \subset Aut^*(R)$ be a finite group acting  on $R$ (see \ref{action}). Let $S = R^G$ be the ring of invariants. If $S$ is \CM \ then $S \in \F_\Oo$.
\end{proposition}
\begin{proof}
  It suffices to prove $H^2_{S_+}(S)_j = 0$ for $j \geq -1$. Let $P = (\pi)S$ and let $\M$ be the graded maximal ideal of $S$.  Recall that  $p H^1(G,R) = 0$ (see \cite[6.5.8]{W}). So support of $H^1(G, R)$  is contained in $V(P)$. It follows that
  $H^0_{S_+}(H^1(G, R)) = H^0_{\M}(H^1(G, R))$ where $\M$ is the $*$-maximal ideal of $S$. The latter module is zero as $S$ is \CM, see \ref{prelim}(4). By \ref{first-es}(1) the result follows.
\end{proof}
\section{Proof of Theorem \ref{equi}}
 \s\label{e-hyp} In this section we prove Theorem \ref{equi}. Throughout  this section we assume that $A$ is the ring of integers in a finite extension of $F_q(t)$ where $F_q$ is a finite field. Say $q = p^s$ for a prime $p$.

 \begin{lemma}\label{ord1}
 (with hypotheses as in \ref{e-hyp}). Let $\Oo = A_\q$ for some  non-zero prime $\q$ in $A$. Let $(\pi)$ be the maximal ideal of $\Oo$. Let $G \subseteq GL_2(\Oo)$ be a $p$-group acting linearly on $R = \Oo[x,y]$.  Set $S = R^G$. Then there exists non-zero $v \in R_1$ such that
 \begin{enumerate}[\rm (1)]
   \item  $v \in S_1$.
   \item $v, \pi$ is an $S$-regular sequence.
   \item $H^2_{S_+}(S/vS) = 0$.
 \end{enumerate}
 \end{lemma}
\begin{proof}
(1) and (2). Let $K = \Oo_\pi$ be the quotient field of $\Oo$. Consider the induced action of $G$ on $T = K[x,y]$. By \cite[8.2.1]{S}, there exists non-zero $u \in T_1$ such that $u \in T^G$.
Let $u = v/\pi^m$ where $v \in R_1$ such that $v \notin \pi R$. Then $u = \sigma(u)$ for all $\sigma \in G$ implies that $v = \sigma(v)$ for all $\sigma \in G$. So $v \in S_1$. Furthermore as $v \notin \pi R$ it follows that $v \notin \pi S$. As $S/\pi S$ is a domain (see \ref{prelim}) it follows that $\pi, v$ is an $S$-regular sequence. As $S$ is *-local it follows that $v, \pi$ is an $S$-regular sequence.

(3) Note that $S/\pi S$ is a graded domain of dimension two with $(S/\pi S)_0 = \Oo/\pi \Oo$ a field. As $v$ is non-zero in $S/\pi S$ it follows that we can find $w$ homogeneous such that $v, w$ is a homogeneous system of parameters of $S/\pi S$. It follows that $H^2_{S_+}(S/(\pi, v)S) = 0$. As $v, \pi$ is a $S$-regular sequence we obtain an exact sequence
\[
0 \rt S/vS \xrightarrow{\pi} S/vS \rt S/(v, \pi)S \rt 0.
\]
This induces an exact sequence in cohomology for all $n \in \Z$,
\[
H^2_{S_+}(S/vS)_n \xrightarrow{\pi} H^2_{S_+}(S/vS)_n \rt H^2_{S_+}(S/(v, \pi)S)_n = 0.
\]
Notice $H^2_{S_+}(S/vS)_n$ is a finitely generated $\Oo$-module (see \cite[15.1.5]{bs}). So by Nakayama Lemma it follows that $H^2_{S_+}(S/vS)_n = 0$ for all $n \in \Z$.
\end{proof}
Next we give
\begin{proof}[Proof of Theorem \ref{equi}] Let $K$ be a Sylow $p$ subgroup of $G$. By Theorem \ref{sylow} it suffices to prove $R^{K} $ is  \CM \ (as for all other primes $l$ dividing $|G|$ note that $l$ is invertible in $\Oo$). So we assume $G$ is a $p$-group. We localize $A$ at a non-zero prime $\p$. Let $(\Oo, (\pi) )= (A_\p, \p A_\p)$. Set $R = \Oo[x,y]$ and $S = R^G$. We have an exact sequence, see \ref{first-es}
\begin{equation*}
 0 \rt H^0_{S_+} (H^1(G, R)) \rt H^2_{S_+}(S) \rt L \rt 0.  \tag{$\dagger$}
\end{equation*}
where $L_{j} = 0$ for $j \geq -1$.
Also note that $H^1(G, R)_j = 0$ for $j \leq -1$, see \ref{non-neg}.
It follows that $H^2_{S_+}(S)_{-1} = 0$.
Let $v \in S_1$ be as in Lemma \ref{ord1}. We have an exact sequence
$$ 0 \rt S(-1) \xrightarrow{v} S \rt S/vS \rt 0. $$
By \ref{ord1} we have $H^2_{S_+}(S/vS) = 0$. By the long exact sequence in cohomology we get for all $n \in \Z$
\begin{equation*}
  H^2_{S_+}(S)_{n -1} \rt  H^2_{S_+}(S)_n \rt 0. \tag{*}
\end{equation*}
Put $n = 0$ in (*). As $H^2_{S_+}(S)_{-1} = 0$ we get $H^2_{S_+}(S)_{0} = 0$.
Put $n = 1$ in (*). As $H^2_{S_+}(S)_{0} = 0$ we get $H^2_{S_+}(S)_{1} = 0$.
Iterating we get  $H^2_{S_+}(S)_{n} = 0$  for all $n \geq -1$.

By ($\dagger$) it follows that $H^0_{S_+}(H^1(G,R)))_n = 0$ for all $n \geq 0$. As $H^1(G, R)_j = 0$ for all $j < 0$, see \ref{non-neg},  we get $H^0_{S_+}(H^1(G, R))) = 0$. So $\depth H^1(G, R) > 0$. By \ref{prelim} it follows that $S$ is \CM.
\end{proof}

\section{Strategy to prove Theorem \ref{mixed}}
In this section we discuss our strategy to prove Theorem \ref{mixed}. We also indicate the motivation for the next two sections.

\s Let $K$ be a number field and let $A$ be the ring of integers of $K$. Let $G$ be a finite subgroup of $GL_2(A)$ acting linearly on $R = A[X,Y]$ (fixing $A$). We assume that for every prime $p$ dividing $|G|$, the Sylow $p$-subgroups have exponent $p$. We  want to prove that $R^G$ is \CM.
We do the following steps to prove this result:

\emph{Step-1} For $p$ dividing $|G|$ let $H_p$ be a Sylow $p$-subgroup of $G$. By Theorem \ref{sylow} it sufices to prove $R^{H_p}$ is \CM \ for every prime $p$ dividing $|G|$. Thus it suffices to assume $G$ is a $p$-group.

\emph{Step-2} Let $G$ be a $p$-group of exponent $p$. By Theorem \ref{red-hilb} it suffices to assume $A$ is the ring of integers of the Hilbert class field of $\mathbb{Q}(e^{2\pi i/p})$.

\emph{Step-3} As $R^G$ is a $\mathbb{N}$-graded ring, it suffices to prove that if $\M$ is a maximal homogeneous ideal of $R^G$ then $(R^G)_\M = (R_\M)^G$ is \CM.  We note that $\M \cap A = P$ a maxiaml ideal of $A$. Thus it suffices to prove that $R^G_P = (R_P)^G$ is \CM.  Let $P \cap \Z = (q)$ where $q$ is a prime. If $q \neq p$ then $|G|$ is invertible in $A_P$. So
$A_P[X, Y]^G$ is \CM. This leaves the essential case when $q = p$.
Thus it suffices to prove the following result.

\begin{theorem}
 \label{sufficient} Let $A$ be the ring of integers of the Hilbert class field of $\mathbb{Q}(e^{2\pi i/p})$. Let $P$ be a prime ideal of $A$ containing $p$. Let $(\Oo, (\pi)) = (A_P, PA_P)$. Let $G \subseteq GL_2(\Oo)$ be a $p$-group acting linearly on $R = \Oo[X,Y]$. Then $R^G$ is \CM.
\end{theorem}
We prove the above result by induction on $|G|$. Although we are only interested in the polynomial ring case,  the intermediate rings that arise in induction are not polynomial rings. So we prove the assertion of Theorem \ref{sufficient} for a more general class of rings $\F_\Oo$ defined in \ref{class} with group actions as defined in \ref{action}.

We first consider the case when $G \subseteq Aut^*(R)$ is cyclic of order $p$. We have to consider two cases

Case (1) The natural map $G \rt Aut^*(R/\pi R)$ is injective.  We cover this case in section \ref{case-1-sect}.

Case (2) The natural map $G \rt Aut^*(R/\pi R)$ is trivial.  We cover this case in section \ref{case-2-sect}.

We will need the following fact about Hilbert class field of  $\mathbb{Q}(e^{2\pi i/p})$.
\begin{lemma}
 \label{units} Let $A$ be the ring of integers of $K$, the Hilbert class field of $\mathbb{Q}(e^{2\pi i/p})$. Let $P$ be a prime ideal of $A$ containing $p$. Let $(\Oo, (\pi)) = (A_P, PA_P)$.  If $\xi$ is any primitive $p^{th}$-root of unity then
 $\xi - 1 = \pi u_\xi$ where $u_\xi$ is a unit in $\Oo$.
\end{lemma}
\begin{proof}
 Let $B$ be the ring of integers of  $\mathbb{Q}(e^{2\pi i/p})$. Fix a primitive $p$-th root of unity $\eta$ in $B$. Set
 $Q = (\eta -1)B$.
 Note
 $(\eta^i -1) = (\eta -1)u_i$ where $u_i$ is a unit in $B$. So we have for $i = 1, \ldots, p-1$ we have $Q = (\eta^i - 1)B$. We also have $pB = Q^{p-1}$ and $Q$ is a prime ideal in $B$. For these facts see \cite[13.2.7]{IR}.

 The extension $K$ is an un-ramified extension of  $\mathbb{Q}(e^{2\pi i/p})$. So we have
 $QA = P_1P_2 \cdots P_g$ where $P_i$ are distinct primes. If $P$ is a prime in $A$ containing $p$ then $P \in \{ P_1, \ldots, P_g\}$ and $QA_P = PA_P$. The result follows.
\end{proof}

\section{Case-1}\label{case-1-sect}
In this section we assume the following
\s \label{setup-case1} $(\Oo,(\pi))$ is a DVR such that $\Oo_\pi$ has characteristic zero and $\Oo/(\pi)$ has characteristic $p >0$.
Let $R \in \F_\Oo$, see \ref{class}.  Let $G \subset Aut^*(R)$ be a group of order $p$ acting  on $R$ (see \ref{action}). Let $S = R^G$ be the ring of invariants.
 We also assume that the natural map $G \rt Aut^*(R/\pi R)$ is injective.
Let $S = R^G$ be the ring of invariants of $G$. In this section we prove
\begin{theorem}\label{case-1}
(with hypotheses as in \ref{setup-case1})
$S \in  \F_\Oo$.
\end{theorem}
\begin{proof}
It suffices to prove $S$ is \CM, see \ref{incl}.
 Let $\ov{R} = R/\pi R$ and $\ov{S} = S/\pi S$. Also let $K$ be the field of fractions of $\ov{R}$.
As $G \subseteq Aut^*(R/\pi R) $ we note that computing group cohomology of $\ov{R}$ as a $S$-module or over $\ov{S}$ yields the same answer.  Let $\sigma$ be a generator of $G$.

Claim-1: $\dim H^1(G,\ov{R}) \leq 1$.\\
Let $W = \ov{R}^{G} \setminus \{ 0 \}$. Then notice $W^{-1}\ov{R} = K$. So $W^{-1}H^1(G, \ov{R}) = H^1(G, K)$, see \ref{localization}. We note that $K$ is a Galois extension of the field $W^{-1}\ov{R}^G$ with Galois group $G$, (see \cite[ Chapter VI, 1.8]{L}). By the additive form of Hilbert Theorem 90, see \cite[Chapter VI, 6.3]{L},  $H^1(G, K)  = 0$.
The result follows.

Claim-2: $\dim H^2(G, R) \leq 1$. \\
Let $P = (\pi)S$. Recall that  $p H^2(G,R) = 0$ (see \cite[6.5.8]{W}). So support of $H^2(G, R)$  is contained in $V(P)$. We assert that $H^2(G, R)_P = 0$.

 By \ref{les} we have an exact sequence $0 \rt E \rt H^2(G, R) \xrightarrow{\pi} H^2(G, R) $ where $E$ is a a quotient of $H^1(G,\ov{R})$ and so has dimension $\leq 1$. In particular $E_P = 0$.
 So we have $ 0 \rt H^2(G, R)_P \xrightarrow{\pi} H^2(G, R)_P$. As $H^2(G, R)_P$ has finite length as an $S_P$ module we get $H^2( G, R)_P = \pi H^2(G, R)_P$. Since $H^2(G, R)_P$ is finitely generated $S_P$-module and $\pi \in P$ we get $H^2(G, R)_P = 0$ by Nakayama's Lemma. So Claim-2 follows.

 By \ref{first-es} we have an exact sequence $$0 \rt  H^0_{S_+}(H^3(G,R)) \rt H^2_{S_+}(H^2(G, R)) = 0.$$
  So $H^0_{S_+}(H^3(G, R)) = 0$. As $G$ is cyclic $H^3(G, R) \cong H^1(G,R)$. It follows that $\depth H^1(G, R) > 0$. So $S$ is \CM \ by \ref{prelim}.
 \end{proof}

\section{Case 2}\label{case-2-sect}
We first discuss the hypotheses on the DVR $\Oo$.
\s \label{setup-case2-O} $(\Oo,(\pi))$ is a DVR such that $\Oo_\pi$ has characteristic zero and $\Oo/(\pi)$ has characteristic $p >0$. We assume that $\Oo$ contains $p^{th}$ roots of unity.  We assume that if $\xi$ is any  primitive $p^{th}$-root of unity then $\xi-1 = \pi u_\xi$ where $u_\xi$ is a unit in $\Oo$.

\begin{remark}
 Let $\Oo$  be a  localization at a prime containing $p$  of the ring of  integers of the Hilbert class field of $\mathbb{Q}(e^{2\pi i/p})$. Then $\Oo$ satisfies the hypotheses of \ref{setup-case2-O}, see \ref{units}.
\end{remark}

\s \label{case-2}  Let $\Oo$ satisfy assumption as in \ref{setup-case2-O}.
Let $R \in \F_\Oo$, see \ref{class}.  Let $G \subset Aut^*(R)$ be a  group of order $p$ acting  on $R$ (see \ref{action}). Let $S = R^G$ be the ring of invariants.
 We also assume that the natural map $G \rt Aut^*(R/\pi R)$ is trivial.
Let $S = R^G$ be the ring of invariants of $G$. \

In this section we prove
\begin{theorem}\label{c2-thm} (with hypotheses as in \ref{case-2}). $S \in \F_\Oo$.
\end{theorem}
We need the following result.
\begin{proposition}\label{g1}
  Let $\Oo$ satisfy hypotheses in \ref{setup-case2-O}. Let $W$ be a finitely generated free $\Oo$-module. Let $G = < \sigma>$ be a group of order $p$ acting on $W$. We assume the induced action on $W/\pi W$ is trivial. Then $\pi H^1(G, W) = 0$.
\end{proposition}
We give a proof of Theorem \ref{c2-thm} assuming Proposition \ref{g1}.
\begin{proof}[ Proof of Theorem \ref{c2-thm}]
It suffices to prove $S$ is \CM, see \ref{incl}.
Set $R =  \bigoplus_{n \geq 0} R_n$. Then $R_n$ is a finite free $\Oo$-module. As the induced action of $G$ on $R/\pi R$ is trivial we get by Proposition \ref{g1} that
$\pi H^1( G, R_n) = 0$ for all $n \geq 0$. As $H^1(G, R) = \bigoplus_{n \geq 0}H^1(G, R_n)$, see \ref{comp},  we get that $\pi H^1(G, R) = 0$.  By exact sequence \ref{les},  we get an inclusion $ H^1(G, R) \hookrightarrow H^1(G, R/\pi R)$.
As the action of $G$ on $R/\pi R$ is trivial we get $H^1(G, R/\pi R) = R/\pi R$, see \ref{trivial}. It follows that $H^1(G, R)$ has positive depth. By \ref{prelim} we get that $R^G$ is \CM.
\end{proof}
To prove Proposition \ref{g1} we need the following result:
\begin{lemma}\label{basis} Let $\Oo$ satisfy hypotheses in \ref{setup-case2-O}. Let $W$ be a finitely generated free $\Oo$-module. Let $G = < \sigma>$ be a group of order $p$ acting on $W$. We assume the induced action on $W/\pi W$ is trivial. Then there exists a $\Oo$ basis $\{ w_1, \ldots, w_s \}$  of $W$ such that $\sigma(w_i) = \xi_i w_i$  for all $i$, where $\xi_i$ is a $p^{th}$-root of unity.
\end{lemma}
We give a proof of Proposition \ref{g1} assuming Lemma \ref{basis}.
\begin{proof}[Proof of Proposition \ref{g1}]
By \ref{basis} $W$ has a basis  $\{ w_1, \dots, w_s \}$ such that  $\sigma(w_i) = \xi_i w_i$  for all $i$, where $\xi_i$ is a $p^{th}$-root of unity.
Let $\beta = [u] \in H^1(G, W)$. Let $u = \sum_{i = 1}^{s} a_i w_i$. Then
$$ \Tr(u) = \sum_{\stackrel{i}{\xi_i = 1}}a_ipw_i = 0.$$
It follows that
$$ u = \sum_{\stackrel{i}{\xi_i \neq 1}} a_i w_i. $$
If $\xi_i \neq 1$ then by our assumption on $\Oo$ we get that there exists unit $c_i$ such that $\xi_i - 1 = \pi c_i$.
Set
$$ \theta = \sum_{\stackrel{i}{\xi_i \neq 1}} \frac{a_i }{c_i}w_i.$$
Then it is  easy to check that $\sigma(\theta) - \theta = \pi u$. The result follows.
\end{proof}
We now give
\begin{proof}[Proof of Lemma \ref{basis}]
We may assume the action of $\sigma$ on $W$ is not-trivial (otherwise there is nothing to prove). Let $K = \Oo_\pi$ which is a field of characteristic zero. We now consider the induced action of $\sigma$ on $W_\pi$. We note that the minimal polynomial of this action divides $t^p - 1$ which has $p$-distinct roots in $K$. So the action of $\sigma$ on $W_\pi$ is diagonalizable. Say $\sigma(e_1) = \xi e_1$ where $e_1 \in W_\pi$ and $\xi$ is a primitive $p^{th}$-root of unity. Say $e_1 = w_1/\pi^l$ where $w_1 \in  W\setminus \pi W$. It follows that $\sigma(w_1) = \xi w_1$.

We prove the result by induction on $s = \rank_\Oo W$. If $s = 1$ then note that $\{ w_1 \}$ is a basis of $W$ with the required properties. We now assume that $s = \rank W \geq 2$ and the result has been proved when rank is $s -1$. Let $W_1 = \Oo w_1$. Then note that $W_1$ is a free summand of $W$. Furthermore $W_1$ is a $\Oo[G]$-submodule of $W$. We have a short exact sequence of $\Oo[G]$-modules $ 0 \rt W_1 \rt W \rt W/W_1 \rt 0$. As $W/W_1$ is a free $\Oo$-module, by induction hypothesis $W/W_1$ has a basis $\{w_2^\prime,\cdots, w_s^\prime \}$ such that
$\sigma(w_i^\prime) = \xi_iw_i^\prime$ where $\xi_i$ is a  $p^{th}$-root of unity. We assume that $\xi_i = \xi$ for $2 \leq r$ and $\xi_i \neq \xi$ for $i  > r$  (we do not exclude the possibility that $r = 1$). So we have a basis $\{ w_1,w_2, \cdots, w_s \}$ of $W$ such that for $j \geq 2$ we have $\sigma(w_j) = \xi_j w_j + a_j w_1$.

We may assume there exists $c \leq r$ such that $\sigma(w_j) = \xi w_j $ for $j =1, \ldots,c$ (we do not exclude the possibility that $c = 1$). If $c < r$, say $\sigma(w_{c+1}) = \xi w_{c+1} + a_{c+1} w_1$. As $K[G]$ is semi-simple, the short exact sequence $0 \rt (W_1)_\pi \rt W_\pi \rt (W/W_1)_\pi \rt 0$ splits. It follows that there exists $v \in W$ such that $v$ is not in $K$-span of
$\{ w_1, \ldots, w_c \}$  and $\sigma(v) = \xi v$. Say
$$v = \sum_{i = 1}^{c}\alpha_i w_i + \sum_{i = c +1}^{r}\alpha_i w_i + \sum_{i > r}\alpha_i w_i. $$
We have
$$\sigma(v) =  \sum_{i = 1}^{c}\xi\alpha_i w_i + \sum_{i = c +1}^{r}\xi \alpha_i w_i + \left(\sum_{i = c +1}^{r}\alpha_i a_i\right)w_1 + \sum_{i > r} \xi_i\alpha_i w_i + \left(\sum_{i > r}\alpha_i a_i\right)w_1.$$
As $\sigma(v) = \xi v$ and $\xi_i \neq \xi$ for $i  > r$ we have that $\alpha_i = 0$ for $i > r$. We also get $\sum_{i = c +1}^{r}\alpha_i a_i = 0$.
As $v$ does not belong to the $K$-span of $\{ w_1, \ldots, w_c \}$ we get that $\alpha_i \neq 0$ for some $i > c$. After permuting $w_i$  we may  assume $\alpha_j \neq 0$ for $j = c+1, \ldots, l$ and $\alpha_j = 0$ for $j > l$. Let $\alpha_j = \pi^{r_j}\beta_j$ with $\beta_j $ a unit. After permuting we may assume $r_{c+1} \leq r_j$ for $j =c+2, \ldots, l$. As
$\sum_{i= c+1}^{l}\pi^{r_i}\beta_i a_i = 0$ it follows that $\sum_{i= c+1}^{l}\pi^{r_i - r_{c+1}}\beta_i a_i = 0$.
Set
$$w^\prime_{c+1} = \beta_{c+1}w_{c+1} + \sum_{i = c+2}^{l}\beta_i  \pi^{r_i - r_{c+1}}w_i.$$
Then note $\sigma(w^\prime_{c+1}) = \xi w^\prime_{c+1}$. Furthermore $\{ w_1, \ldots, w_c, w^\prime_{c+1}, w_{c+2}, \ldots, w_s \}$ is still a basis of $W$. Iterating we may assume that we have a basis $\{ w_1, \ldots, w_s \}$ of $W$ such that $\sigma(w_i) =\xi w_i$ for $i \leq r$ and $\sigma(w_i) = \xi_i w_i + a_i w_1$ for $i > r$ (where $\xi_i \neq \xi$).

Assume $a_i = 0$ for $i = r+1, \ldots, q$ and $a_{q+1} \neq 0$. As the action of $G$ on $W/\pi W$ is trivial we get that $a_{q+1} \in (\pi)$. Say $a_{q+1} = \gamma \pi^t$.
We note that $\xi - \xi_{q+1} = \pi u$ where $u$ is a unit in $\Oo$. Set $d = \pi^{t-1}\gamma/u$. Set $w_{q+1}^\prime = dw_1 + w_{q + 1}$. It is readily verified that $\sigma(w_{q+1}^\prime) = \xi_{q +1} w_{q+1}^\prime$. Furthermore $\{ w_1, \ldots, w_{q}, w^\prime_{q+1}, w_{q+2}, \cdots, w_r \}$ is still a basis of $W$. Iterating we obtain the required basis of $W$.
\end{proof}

\section{Proof of Theorem \ref{mixed}  }
In this section we give a proof of Theorem \ref{mixed}. We first prove
\begin{theorem}
\label{mix-local}Let $\Oo$ satisfy assumption as in \ref{setup-case2-O}.
Let $R \in \F_\Oo$, see \ref{class}.  Let $G \subset Aut^*(R)$ be a  $p$-group  acting  on $R$ (see \ref{action}). Let $S = R^G$ be the ring of invariants.
Then $S \in \F_\Oo$.
\end{theorem}
\begin{proof}
 We prove the result by induction on order of $G$.

We first consider the case when $|G| = p$. Consider the natural map $\eta \colon G \rt Aut^*(R/\pi R)$. Then either $\ker \eta = G$ or it is trivial.
In the first case $S$ is in $\F_\Oo$ by Theorem \ref{c2-thm}. In the second case $S$ is in $\F_\Oo$ by Theorem \ref{case-1}.

Now assume $|G| = p^r$ and the assertion is proved for all  $p$-groups of order $\leq p^{r-1}$. Consider the natural map $\eta \colon G \rt Aut^*(R/\pi R)$.\\
Case 1: $K = \ker \eta$ is a proper subgroup of $G$. \\
By induction hypotheses $R^K$ is \CM \ and so $R^K \in \F_\Oo$. We note that $G/K$ is a finite subgroup of $Aut^*(R^K)$. So again by induction hypothesis  $S = (R^K)^{G/K} \in \F_\Oo$.

Case 2: $K = \ker \eta = \{ 1 \}$. \\
Let $H$ be a subgroup of order $p$ contained in the center of $G$. Then $H$ is a normal subgroup of $G$.  By Theorem \ref{case-1},  $R^H$ is in $\F_\Oo$. We note that $G/H$ is a finite subgroup of $Aut^*(R^H)$. So again by induction hypothesis  $S = (R^H)^{G/H} \in \F_\Oo$.

Case 3: $K = \ker \eta = G$. \\
Let $H$ be a subgroup of order $p$ contained in the center of $G$. Then $H$ is a normal subgroup of $G$.  By Theorem \ref{c2-thm},  $R^H$ is in $\F_\Oo$. We note that $G/H$ is a finite subgroup of $Aut^*(R^H)$. So again by induction hypothesis  $S = (R^H)^{G/H} \in \F_\Oo$.
\end{proof}

We now give:
\begin{proof}[Proof of Theorem \ref{mixed}]
By Theorem \ref{sylow} it suffices to prove that for every prime dividing order of $G$ the ring $R^{Syl_p(G)}$ is \CM. So it suffices to assume that $G$ is a $p$-group.
By \ref{red-hilb}
we may assume that  $A$ is  the ring of  integers of the Hilbert class field of $\mathbb{Q}(e^{2\pi i/p})$. Let $P$ a height one prime in $A$. It suffices to prove $A_P[X,Y]^G$ is \CM. Let $P \cap \Z = (q)$ where $q$ is prime. If $q \neq p$ the $|G|$ is invertible in $A_P$. In this case it is well-known that $A_P[X,Y]^G$ is \CM.
If $q = p$ then  $A_P = \Oo$ satisfies the hypotheses of \ref{setup-case2-O}, see \ref{units}. By Theorem \ref{mix-local}, $A_P[X, Y]^G$ is \CM. The result follows.
\end{proof}

\section{Proof of Theorem \ref{unramify}}
In this section we prove Theorem \ref{unramify}. We need some preliminaries from algebraic number theory.

\s \label{efg} Let $K \subseteq  L$ be an extension of number fields and let $A \subseteq B$ be the corresponding extension of their ring of integers.
Let $\p$ be a prime in $A$. Let $\p B = P_1^{e_1}P_2^{e_2} \cdots P_g^{e_g}$ where $P_i$ are distinct primes in $B$ and $e_i \geq 1$. Recall $\p$ is said to be un-ramified in $B$ if each $e_i = 1$. The number $e_i$ is denoted as $e(P_i/\p)$. Let $f(P_i/\p) = [B/P_i \colon A/\p]$ be the degree of the field extension $B/P_i$ of $A/\p$. Then by \cite[Proposition 21, p.\ 24]{L-ant} we have
\[
\sum_{i = 1}^{g}e(P_i/\p)f(P_i/\p) = \dim_K L.
\]
The next result is a crucial ingredient to prove Theorem \ref{unramify}.
\begin{theorem}\label{ing}
Let $p$ be a prime number and let $K$ be a number field. Assume $p$ is un-ramified in $K$. Let $\zeta$ be a primitive $p^{th}$-root of unity and let $\q$ be the unique prime in $\mathbb{Q}(\zeta)$ lying above $p$. Then $\q$ is un-ramified in $K(\zeta)$.
\end{theorem}
\begin{proof}
  \emph{Claim-1:} $K \cap \mathbb{Q}(\zeta) = \mathbb{Q}$.\\
  Set $E = K \cap \mathbb{Q}(\zeta)$ and let $r = \dim_{\mathbb{Q}} E$. Let $D$ be the  ring of integers of $E$. We note that as there is only one prime lying above $p$ in $\mathbb{Q}(\zeta)$ there will be only one prime in $E$ lying above $p$. Also as $p$ is un-ramified in $K$ it is un-ramified in $E$, see \cite[Proposition 20, p.\ 24]{L-ant}. Let $pD = P$.
  Note $1 = f(\q/p) = f(\q/P)f(P/p)$. So $f(P/p) = 1$. By \ref{efg} we get $r = 1$.

  Let $s = \dim_\mathbb{Q} K$. We note that $\dim_K K(\zeta) = p -1$, see \cite[Chapter V1, 1.12]{L}. It follows that $\dim_{\mathbb{Q}(\zeta)} K(\zeta) = s$.
  Let $A, B, C$ be the ring of integers of $K, \mathbb{Q}(\zeta), K(\zeta)$ respectively. Let $pB = P_1P_2\cdots P_g$. Let $Q_1, \ldots, Q_m$ be the primes in $C$ lying above $\q$. Then note $m \geq g$. We may assume (after re-indexing) that $Q_i\cap A = P_i$. Let
  $$\q C = Q_1^{e_1} Q_2^{e_2} \cdots Q_m^{e_m}. $$
  Then
  $$p C = \q^{p-1} C  =  Q_1^{e_1(p-1)} Q_2^{e_2(p-1)} \cdots Q_m^{e_m(p-1)}. $$
  By \ref{efg} we get
$$s(p-1) = \sum_{i =1}^{m} e_i(p-1)f(Q_i/p). $$
So we have $s =  \sum_{i =1}^{m} e_if(Q_i/p)$. As $p$ is un-ramified in $K$ we obtain \\
$s = \sum_{i =1}^{g} f(P_i/p)$. As $Q_i \cap A = P_i$ it follows that $f_i(P_i/p)$ divides $f(Q_i/p)$. It follows that $m = g$ and $e_i = 1$ for all $i$. So $\q$ is unramified in $K(\zeta)$.
\end{proof}

\begin{lemma}\label{unite-ur}(with hypotheses as in \ref{ing}) Let $C$ be ring of integers in $K(\zeta)$ and let $P$ be a prime in $C$ containing $p$.  Let $(\Oo, (\pi)) = (C_P, PC_P)$.  If $\xi$ is any primitive $p^{th}$-root of unity then
 $\xi - 1 = \pi u_\xi$ where $u_\xi$ is a unit in $\Oo$.
\end{lemma}
The proof of the above Lemma is similar to  Lemma \ref{units} and so is omitted. Next we show
\begin{theorem}\label{unram-local}(with hypotheses as in \ref{unite-ur}) Let $R \in \F_\Oo$, see \ref{class}.  Let $G \subset Aut^*(R)$ be a  $p$-group  acting  on $R$ (see \ref{action}). Let $S = R^G$ be the ring of invariants.
Then $S \in \F_\Oo$.
\end{theorem}
\begin{proof}
  By \ref{unite-ur} we get that $\Oo$ satisfies the hypotheses as in \ref{setup-case2-O}. The result follows from Theorem \ref{mix-local}.
\end{proof}
We now give
\begin{proof}[ Proof of Theorem \ref{unramify}]
By Theorem \ref{sylow} it suffices to prove that for every prime dividing order of $G$ the ring $R^{Syl_p(G)}$ is \CM. So it suffices to assume that $G$ is a $p$-group.
Let $\zeta$ be a primitive $p^{th}$-root of unity. Set $L =K(\zeta)$ and let $C$ be the ring of integers of $L$. By \ref{bc-prop} it suffices to prove $C[X, Y]^G$ is \CM. Let
$P$ be a prime in $C$ and set $\Oo = C_P$. Set $(q) = P \cap \Z$. If $q \neq p$ then $p$ is a unit in $\Oo$. In this case it is well known that $\Oo[X, Y]^G$ is \CM. If $q = p$ then by \ref{unram-local} it follows that $\Oo[X, Y]^G$ is \CM. The result follows.
\end{proof}
\section{Cyclic groups}
In this section we prove Theorem \ref{cyclic-thm}.
The following result is a crucial ingredient to prove it.
\begin{lemma}\label{cyclic-lemma}
Let $(\Oo, \pi)$ be a DVR of mixed characteristic. Let $G = <\sigma > \subseteq GL_2(\Oo)$ be a cyclic group of order $m \geq 2$ such that there exists a basis ${X, Y}$ of $\Oo^2$ with  $\sigma(X) = \epsilon_1 X$ and $\sigma(Y) = \epsilon_2 Y$. Let $G$ act linearly on $\Oo[X, Y]$ (fixing $\Oo$). Then $R^G$ is
\CM.
\end{lemma}
\begin{proof}
We note that $\epsilon_i^m = 1$ for $i = 1, 2$. It follows that $X^m, Y^m \in S = R^G$.

\emph{Claim-1:} $X^m, Y^m$ is a $S$-regular sequence. \\
We have an exact sequence $0 \rt S \rt R \xrightarrow{\sigma - 1} R$ and $X^m, Y^n$ is an $R$-regular sequence.
The assertion follows.

\emph{Claim-2} If $u = \sum_{i,j}a_{ij}X^iY^j \in S$ then $X^iY^j \in S$ for all $i,j$ with $a_{ij} \neq 0$.

Note $u = \sigma(u) = \sum_{i,j}a_{ij}\epsilon_1^i \epsilon_2^j X^iY^j $. Comparing coefficients the result follows.

We now assert that $X^m, Y^m, \pi$ is a $S$-regular sequence. This will prove our result.
By Claim-1 we have $X^m, Y^m$ is a $S$-regular sequence.  Let $u \in S$ such that $\pi u \in (X^m, Y^m)S$. We want to show $u \in (X^m, Y^m)S$.  Let $u = \sum_{i,j}a_{ij}X^iY^j \in S$.  Then $X^iY^j \in S$ for all $i,j$ with $a_{ij} \neq 0$.  Furthermore either $i \geq m$ or $j \geq m$ by our assumption on $u$.\\
Case-1 $i \geq m$.\\
We have $w = X^iY^j = X^m X^{i-m}Y^j$. As $\sigma(w) = w$ and $\sigma(X^m) = X^m$ it follows that  $t = X^{i-m}Y^j \in S$. So $w \in (X^m, Y^m)S$.\\
Case-2: $j \geq m$.\\
This is similar to Case-1. Again we have $ X^iY^j \in (X^m, Y^m)S$.

Thus $u \in  (X^m, Y^m)S$. So  $X^m, Y^m, \pi$ is a $S$-regular sequence. The result follows.
\end{proof}
We now give
\begin{proof}[Proof of Theorem \ref{cyclic-thm}(1)]
 Let $\zeta$ be a primitive $m^{th}$-root of unity. Set $L =K(\zeta)$ and let $C$ be the ring of integers of $L$. By \ref{bc-prop} it suffices to prove $C[X, Y]^G$ is \CM. Let
$P$ be a prime in $C$ and set $(\Oo , (\pi)) = (C_P, PC_P)$.
Note $\sigma_\pi \colon \Oo^2_\pi \rt \Oo^2_\pi$ is  diagonalizable.
 We can find basis $v_1, v_2$ of $\Oo_\pi^2$ with $\sigma(v_i) = \epsilon v_i$. Let $v_1 = u_1/\pi^s$ with $u_1 \in \Oo^2 \setminus \pi \Oo^2$. Then  $\sigma(u_1) = \epsilon u_1$.
Let $\{ u_1, u_2 \}$ be a basis of $\Oo^2$. We have $v_2 = au_1 + bu_2$ for some $a, b \in \Oo_\pi$. Note $b \neq 0$.
As $\sigma(v_2) = \epsilon v_2$ it follows that   $\sigma(u_2) = \epsilon u_2$.
By \ref{cyclic-lemma} it follows that $\Oo[X,Y]^G$ is \CM. The result follows.
\end{proof}

To prove (2) of Theorem \ref{cyclic-thm} we need the following fact.
\s \label{fact-pr}
Let $p$ be an odd prime and let $D$ be the ring of integers in $\mathbb{Q}(e^{2\pi i/p^r})$. Let $\zeta$ be a primitive $p^{r}$-th root of unity. Let $Q$ be the unique prime in $D$ lying above
$p$. Then $Q = (\zeta - 1)$, see \cite[p.\ 73]{L-ant}. Let $K$ be the Hilbert class field of  $\mathbb{Q}(e^{2\pi i/p^r})$ (with ring of integers $C$). Then $Q$ is un-ramified in $K$. Let $P$ be a prime in
$C$  lying above $p$. Then $PC_P = (\zeta - 1)C_P$.

\begin{proof}[Proof of Theorem \ref{cyclic-thm}(2)]
 By Theorem \ref{sylow} it suffices to   prove that $R^{Sylow_p}$ is \\  \CM \ for each prime $p$ dividing $|G|$.

When $p = 2$ by our hypotheses $Sylow_2(G) = \Z/2 \Z$. So in this case the result follows from Theorem \ref{mixed}.

Now assume $p \neq 2$ and $|H| = p^r$ and $H = <\tau >$ is a Sylow $p$-subgroup of $G$.  We prove $R^H$ is \CM \ by induction on $r$. When $r = 1$ the result follows from Theorem \ref{mixed}. By an argument similar to \ref{red-hilb} it suffices to assume that $A$ is the ring of integers of  Hilbert class field of $\mathbb{Q}(e^{2\pi i/p^r})$. We note that
eigen-values of $\tau$ will be $\epsilon$ and $\epsilon^{-1}$ where $\epsilon$ is a primitive $p^r-{th}$ root of unity. Set $K = <\tau^p>$. Then $|K| = p^{r-1}$ and so by induction hypothesis we have $R^K$ is \CM.  Let
$P$ be a prime in $A$ and set $\Oo = A_P$. Set $(q) = P \cap \Z$. If $q \neq p$ then $p$ is a unit in $\Oo$. In this case it is well known that $\Oo[X, Y]^G$ is \CM.
Now assume $q = p$. Set $T = \Oo[X, Y]$.

Claim-1: There exists a basis $\{ u, v \}$ of $\Oo^2$ such that with respect to this basis
\[
\tau = \begin{pmatrix}
\epsilon & 0 \\ 0 & \epsilon^{-1}
\end{pmatrix}
 \quad \text{or} \quad
 \tau = \begin{pmatrix}
\epsilon & 1 \\ 0 & \epsilon^{-1}
\end{pmatrix}
\]
It is elementary that we have a basis $\{ w_1, w_2 \}$ of $\Oo^2$ such that with respect to this basis
\[
\tau = \begin{pmatrix}
\epsilon & a \\ 0 & \epsilon^{-1}
\end{pmatrix}
\]
We note that $$\epsilon - \epsilon^{-1} = \epsilon^{-2}(\epsilon^2 - 1) = \pi \theta \quad \text{where $\theta$ is a unit}.$$
(We have used above that  $p$ is odd and \ref{fact-pr}). Suppose $a$ is not zero and neither a unit. Say $a = \pi^r \alpha$ where $\alpha$ is a unit.
Set
\[
\gamma = - \frac{\pi^{r}\alpha}{\epsilon - \epsilon^{-1} } \quad \text{and} \quad w_2^\prime = \gamma w_1 + w_2.
\]
Then check that $\tau(w_2^\prime) = \epsilon^{-1}w_3^\prime$.
Furthermore
$\{w_1, w_2^\prime \}$ is a basis of $\Oo^2$ and with respect to this basis
\[
\tau = \begin{pmatrix}
\epsilon & 0 \\ 0 & \epsilon^{-1}
\end{pmatrix}
\]
Thus the only case we have to consider is when $a$ is a unit. Then note $\{ aw_1, w_2 \}$ is a basis of $\Oo^2$ and with respect to this basis
\[
\tau = \begin{pmatrix}
\epsilon & 1 \\ 0 & \epsilon^{-1}
\end{pmatrix}
\]
Thus Claim-1 is proved.

If $\tau$ with respect to some basis of $\Oo^2$ is
\[
\tau = \begin{pmatrix}
\epsilon & 0 \\ 0 & \epsilon^{-1}
\end{pmatrix}
\]
then by \ref{cyclic-lemma} we ger $T^G$ is \CM.

We consider the second case of Claim-1. We may without loss of any generality assume that with respect to basis $\{ X, Y \}$ of $\Oo^2$ we have
\[
\tau = \begin{pmatrix}
\epsilon & 1 \\ 0 & \epsilon^{-1}
\end{pmatrix}
\]
By induction hypothesis $T^K$ is \CM. We note that $T^K \in \F_\Oo$, see \ref{incl}.
We have in $R/\pi R$ that $\tau(\ov{Y)} =  \ov{X} + \ov{Y}$.
Set
$$ t =  \prod_{i = 0}^{p^{r-1} - 1}\tau^{pi}(Y).$$
 Then note that $t \in R^K$.
 As order of $\ov{\tau}$ is $p$ we get
  $$\ov{t} = \ov{Y}^{p^{r-1}} \quad \text{and} \quad \tau(\ov{t}) = \ov{X}^{p^{r-1}} + \ov{Y}^{p^{r-1}}.$$
 As $H/K$ is cyclic of order $p$ it follows that the natural map $H/K \rt Aut^*(T^K/\pi T^K)$ is injective. By \ref{case-1} it follows that $(T^K)^{H/K}$ is \CM.
 It remains to note that $T^H = (T^K)^{H/K}$.
\end{proof}

\emph{Acknowledgement:} I thank Sudhir Ghorpade and Dipendra Prasad for many discussions on cyclotomic extensions.


\begin{thebibliography} {99}

\bibitem{A}
A.~Almuhaimeed
\emph{The Cohen-Macaulay property of invariant rings over the integers},
Transformation Groups,
Vol. 27, No. 2, 2022, pp. 343–369


\bibitem{CR}
C.~W.~Curtis and  I.~Reiner,
\emph{Representation Theory of Finite Groups and Associative Algebras},
Pure Appl. Math., vol. XI, Interscience Publishers, a division of John Wiley and Sons, New York--London, 1962.

\bibitem{ES}
 G.~Ellingsrud, T.~Skjelbred,
 \emph{Profondeur d'anneaux d'invariants en caract\'{e}ristique p} (French),
  Compositio Math. 41 (1980), no. 2, 233--244.




\bibitem {bs}  M. Brodmann and R. Y. Sharp,
\emph{Local Cohomology: An algebraic introduction with geometric applications},
Cambridge Studies in Advanced Mathematics, 60.
Cambridge U. Press, 1998.

\bibitem {BH}  W.~Bruns and J.~Herzog,
\emph{Cohen-Macaulay Rings}, revised edition,
Cambridge Studies in Advanced Mathematics, 39.
Cambridge University Press, 1998.

\bibitem{IR}
 K.~Ireland and  M.~Rosen,
 \emph{A classical introduction to modern number theory},
 Second edition. Graduate Texts in Mathematics, 84. Springer-Verlag, New York, 1990.

 \bibitem{L-ant}
S.~Lang,
\emph{Algebraic number theory},
Second edition
Grad. Texts in Math., 110
Springer-Verlag, New York, 1994.


\bibitem{L}
S.~Lang,
\emph{Algebra},
 Revised third edition. Graduate Texts in Mathematics, 211. Springer-Verlag, New York, 2002.


\bibitem{Lorenz}
 M.~Lorenz,
 \emph{Multiplicative invariant theory},
  Encyclopaedia of Mathematical Sciences, 135. Invariant Theory and Algebraic Transformation Groups, VI. Springer-Verlag, Berlin, 2005.

\bibitem{M}
  H.~Matsumura,
  \emph{Commutative Ring Theory},
   Cambridge University Press, Cambridge (1986) MR0879273


  \bibitem{S}
  L.~Smith,
\emph{Polynomial invariants of finite groups},
Res. Notes Math., 6
A K Peters, Ltd., Wellesley, MA, 1995.

  \bibitem{W}
  C.~A.~Weibel,
  \emph{An introduction to homological algebra},
Cambridge Stud. Adv. Math., 38
Cambridge University Press, Cambridge, 1994.


\end{thebibliography}
\end{document}